\documentclass[11pt,a4paper]{article}
\usepackage{amsfonts,epsf,amsmath,amssymb,graphicx,tabularx,booktabs}

\newtheorem{theorem}{\bf Theorem}[section]

\newtheorem{proposition}[theorem]{\bf Proposition}

\newtheorem{remark}[theorem]{\bf Remark}

\newcommand{\proof}{\noindent{\bf Proof.\ }}
\newcommand{\qed}{\hfill $\blacksquare$ \bigskip}

\begin{document}

\title{\bf{Note on PI and Szeged indices}}

\author{
Aleksandar Ili\' c \\
Faculty of Sciences and Mathematics \\
University of Ni\v s, Serbia \\
e-mail: \tt{aleksandari@gmail.com} \\
}

\date{\today}

\maketitle

\begin{abstract}
In theoretical chemistry molecular structure descriptors are used for modeling physico-chemical,
pharmacologic, toxicologic, biological and other properties of chemical compounds. In this paper we
study distance-based graph invariants and present some improved and corrected sharp inequalities
for PI, vertex PI, Szeged and edge Szeged topological indices, involving the number of vertices and
edges, the diameter, the number of triangles and the Zagreb indices. In addition, we give a
complete characterization of the extremal graphs.
\end{abstract}

 \vskip 0.1cm

%%%%%%%%%%%%%%%%%%%%%%%%%%%%%%%%%%%%%%%%%%%%%%%%%%%%%%%
\section{Introduction}
\label{sec:intro}
%%%%%%%%%%%%%%%%%%%%%%%%%%%%%%%%%%%%%%%%%%%%%%%%%%%%%%%

Let $G = (V, E)$ be a connected simple graph with $n = |V|$ vertices and $m = |E|$ edges. For
vertices $u, v \in V$, the distance $d (u, v)$ is defined as the length of the shortest path
between $u$ and $v$ in $G$. The diameter $diam (G)$ is the greatest distance between two vertices
of $G$. The distance between the vertex $w$ and the edge $e = uv$ is defined as $d' (w, e) = \min
(d (w, u), d (w, v))$.

In theoretical chemistry molecular structure descriptors (also called topological indices) are used
for modeling physico-chemical, pharmacologic, toxicologic, biological and other properties of
chemical compounds \cite{GuPo86}. There exist several types of such indices, especially those based
on vertex and edge distances. Arguably the best known of these indices is the Wiener index $W$,
defined as the sum of distances between all pairs of vertices of the molecular graph
\cite{DoEnGu01, Ro02},
$$
W (G) = \sum_{u, v \in V} d (u, v).
$$

Besides of use in chemistry, it was independently studied due to its relevance in social science,
architecture and graph theory. With considerable success in chemical graph theory, various
extensions and generalizations of the Wiener index are recently put forward \cite{LuZh09,ToCo00}.

Let $e = uv$ be an edge of the graph $G$. The number of vertices of $G$ whose distance to the
vertex $u$ is smaller than the distance to the vertex $v$ is denoted by $n_u (e)$. Analogously,
$n_v (e)$ is the number of vertices of $G$ whose distance to the vertex $v$ is smaller than the
distance to the vertex $u$. Similarly, $m_u (e)$ denotes the number of edges of $G$ whose distance
to the vertex $u$ is smaller than the distance to the vertex $v$. We now define four topological
indices: PI, vertex PI, Szeged and edge Szeged indices of $G$ as follows

% \cite{ChWu09,DaGu10,DaGu09, FaArMoAs10, GuPo86, GuDo98, Ha10, KhKaAg00, KhAzAsWa09, KhYoAs08, KhYo08, KlRaGu96,NaFaAs09}
\begin{eqnarray*}
PI (G) &=& \sum_{e \in E} m_u (e) + m_v (e) \qquad \cite{Ha10,KhKaAg00,KhYoAs08}\\
PI_v (G) &=& \sum_{e \in E} n_u (e) + n_v (e) \qquad \cite{DaGu10,Il10,KhYo08,MoMaAs10,NaFaAs09}\\
Sz (G) &=& \sum_{e \in E} n_u (e) \cdot n_v (e) \qquad \cite{DaGu09,GuDo98,IlKlMi10,KlRaGu96}\\
Sz_e (G) &=& \sum_{e \in E} m_u (e) \cdot m_v (e) \qquad \cite{ChWu09,FaArMoAs10,KhAzAsWa09}.
\end{eqnarray*}

Notice that for trees $W (G) = Sz (G)$ and for bipartite graphs $PI_v (G) = nm$.

The paper is organized as follows. In Section 2 we present two improved inequalities on $PI_v$,
$PI$, $Sz$ and $Sz_e$ indices, and completely describe the extremal graphs. In Section 3 we prove
sharp lower bounds on the vertex PI and Szeged index involving Zagreb indices and correct the
equality cases for the upper bound involving the number of triangles. In Section~4 we correct the
equality case regarding PI and edge Szeged index and present new sharp bound using P\' olya--Szeg\"
o inequality.

%%%%%%%%%%%%%%%%%%%%%%%%%%%%%%%%%%%%%%%%%%%%%%%%%%%%%%%
\section{Improved inequalities for PI and Szeged indices}
\label{sec:2}
%%%%%%%%%%%%%%%%%%%%%%%%%%%%%%%%%%%%%%%%%%%%%%%%%%%%%%%

Let $X_n$ be the set of graphs on $n$ vertices, such that for every edge $e = uv \in E (G)$ it
holds $\min (n_v (e), n_u (e)) = 1$. It is obvious that the complete graph $K_n$ belongs to $X_n$,
and we will exclude $K_n$ in the sequel. By simple calculation, $PI_v (K_n) = 2 |E (K_n)| = n (n -
1)$ and $Sz (G) = |E (K_n)| = \frac{n(n-1)}{2}$.

A chordal graph is a simple graph such that each of its cycles of four or more vertices has a
chord, which is an edge joining two vertices that are not adjacent in the cycle.

In \cite{FaArMoAs10} the authors stated that a graph $G$ from $X_n$ must be a complete graph or a
chordal graph of diameter $2$. It follows that the set $X_n$ is composed of the graphs with
diameter $2$ such that there are no induced path $P_4$ or cycle $C_4$ in the graph $G$. This is the
full characterization of graphs from $X_n$.

Namely, let $G$ be a graph with no induced $P_4$ or $C_4$. It follows that diameter of $G$ is less
than or equal to two. Since $K_n$ is the unique graph with diameter one, we can assume that
diameter of $G$ is equal to 2. Consider an arbitrary edge $e = uv$. Since $G$ does not contain
induced $P_4$ or $C_4$, there are no two vertices $u'$ and $v'$ such that $u'$ is a neighbor of
$u$, $v'$ is a neighbor of $v$, $d (v, u') > 1$ and $d (u, v')
> 1$. Since $diam (G) = 2$, for the vertices $w$ that are not adjacent with $u$ or $v$, it
holds $d (v, w) = d (u, w) = 2$. Finally, either $n_v (e) = 1$ or $n_u (e) = 1$.

\begin{theorem}
\label{thm:pisz} Let $G$ be a connected graph with $n$ vertices and $m$ edges. Then
$$
PI_v (G) \leq Sz (G) + m,
$$
with equality if and only if $G \in X_n$.
\end{theorem}

\proof For an arbitrary edge $e = uv$, we have the following inequality
$$
n_v (e) + n_u (e) \leq n_v (e) \cdot n_u (e) + 1,
$$
which is equivalent with $(n_v (e) - 1)(n_u (e) - 1) \geq 0$. By adding similar inequalities for
all edges $e \in E (G)$, we get $PI_v (G) \leq Sz (G) + m$. The equality holds if and only if $n_v
(e) = 1$ or $n_u (e) = 1$ for all edges $e \in E (G)$, i.e. $G \in X_n$. \qed

\begin{remark}
The authors in \cite{FaArMoAs10} proved the inequality $PI_v (G) \leq 2 Sz (G)$. The inequality
from Theorem \ref{thm:pisz} is stronger, since $Sz (G) \geq m$.
\end{remark}

\begin{remark}
Let $G$ be an arbitrary graph from $X_n$. It can be observed that $G$ contains a vertex with degree
$n - 1$. Namely, consider a vertex $v$ with the maximum degree $k < n - 1$. Let $v_1, v_2, \ldots,
v_k$ be the neighbors of $v$, and assume that $u$ is not adjacent to $v$. Since the diameter of $G$
is equal to $2$, some of the neighbors of $v$ are adjacent to $u$ -- and let one such vertex be
$v_i$. In this case, for $e = vv_i$ it follows $1 = d (v_i, u) < d (v, u) = 2$, and $n_{v_i} (e)
\geq 2$. Therefore, $v_i$ must be adjacent to all vertices $v_1, v_2, \ldots, v_p$, which implies
that $deg (v_i) > deg (v) = k$. This is impossible, and it follows that $v$ is adjacent to all
vertices from $G$.
\end{remark}

Let $Y_n$ be the set of graphs on $n$ vertices, such that for every edge $e = uv \in E (G)$ it
holds $\min (m_v (e), m_u (e)) = 1$. It is obvious that no graph from $Y_n$ contains pendent vertex
(for a pendent vertex $v$ with the only neighbor $u$ it holds $m_v (uv) = 0$). Therefore, the
minimum vertex degree of graphs from $Y_n$ is greater than or equal to $2$. If all vertices have
degree $2$, then $G \cong C_n$ and it can be easily seen that only $C_3$ and $C_4$ belong to $Y_n$.

The vertex $v$ is called branching if $deg (v) > 2$. Let $G$ be an arbitrary graph from $Y_n$ and
let $v$ be an arbitrary branching vertex. For the edge $vu$ we have $m_v (e) \geq 2$, and it
follows that $m_u (e) = 1$. Therefore, all neighbors of the branching vertices have degree two. If
$G$ contains exactly one branching vertex $v$, then $G$ is composed of the union of cycles $C_3$
and $C_4$ having the vertex $v$ in common.

Now assume that $G$ contains at least two branching vertices. Let $P_d = w_0 w_1 \ldots w_d$ be the
shortest path connecting vertices $v$ and $u$, such that $w_0 = v$, $w_d = u$, $deg (v) \geq 3$ and
$deg (u) \geq 3$. If $d > 2$, for the edge $e = vw_1$ we have $m_v (e) \geq 2$ and $m_{w_1} \geq 2$
(since the edge $w_2 w_3$ is closer to $u$ than to $v$). Therefore, the distance between any two
branching vertices is equal to two. If $G$ contains at least three branching vertices, the
contradiction follows by considering the edge $e = vu'$ on Fig. 1 (red edges are closer to $u'$
than to $v$). Finally, in this case $G$ contains exactly two branching vertices connected by paths
of length two.

The set $Y_{12}$ is presented on Fig. 2.

\begin{figure}[h]
  \center
  \includegraphics [width = 4cm]{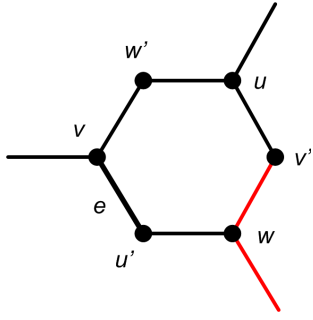}
  \caption { \textit{ Contradiction with three branching vertices. } }
\end{figure}

It can be easily proved by mathematical induction that the number of graphs in $Y_n$ is equal to
$$
|Y_n| = \left\{
\begin{array}{ll}
0 & n = 1, 2 \\
1 & n = 3, 4 \\
\lfloor \frac{n - 1}{6} \rfloor + 1, & n \equiv 2 \pmod 6, \\
\lfloor \frac{n - 1}{6} \rfloor + 2, & otherwise.
\end{array}
\right.
$$

\begin{figure}[ht]
  \center
  \includegraphics [width = 11cm]{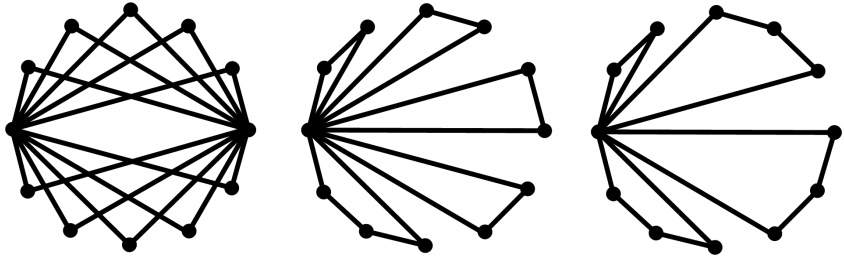}
  \caption { \textit{ The set $Y_{12}$. } }
\end{figure}

\begin{remark}
The authors in \cite{FaArMoAs10} wrongly stated in Theorem 2 that if $\min (m_u (e), m_v (e)) = 1$
then $G$ is a cycle of length $\leq 4$.
\end{remark}

Let $\delta (G)$ denote the minimal vertex degree in the graph $G$.

\begin{theorem}
Let $G$ be a connected graph with $n$ vertices, $m$ edges and $\delta (G) \geq 2$. Then
$$
PI (G) \leq Sz_e (G) + m,
$$
with equality if and only if $G \in Y_n$.
\end{theorem}

\proof The proof is similar to those of Theorem \ref{thm:pisz}. The equality holds if and only if
$m_v (e) = 1$ or $m_u (e) = 1$ holds for all edges $e \in E (G)$, or equivalently $G \in Y_n$. \qed

%%%%%%%%%%%%%%%%%%%%%%%%%%%%%%%%%%%%%%%%%%%%%%%%%%%%%%%
\section{Sharp bounds involving Zagreb indices and number of triangles}
\label{sec:3}
%%%%%%%%%%%%%%%%%%%%%%%%%%%%%%%%%%%%%%%%%%%%%%%%%%%%%%%

One of the oldest graph invariants are the first and the second Zagreb indices
\cite{GuPo86,IlSt10}, defined as follows
\begin{eqnarray*}
M_1(G) &=& \sum_{u \in V(G)} deg(v)^2 \\
M_2(G) &=& \sum_{uv \in E(G)} deg(u) deg(v).
\end{eqnarray*}

Let $t (G)$ denote the number of triangles $K_3$ in the graph $G$.

\begin{proposition}
\label{prp:PI-diameter} Let $G$ be a connected graph with diameter $2$. Then,
$$
PI_v (G) = M_1 (G) - 6 t(G).
$$
\end{proposition}

\proof Let $e = uv$ be an arbitrary edge, such that it belongs to exactly $t (e)$ triangles.
Adjacent vertices of $v$ that are not neighbors of $u$ are closer to $v$ than $u$, and vice versa.
For the vertices $w$ that are not neighbors of $u$ or $v$, it holds $d (v, w) = d (u, w) = 2$.
Therefore,
\begin{eqnarray*}
PI_v (G) &=& \sum_{e \in E (G)} n_v (e) + n_u (e) \\
&=& \sum_{e \in E (G)} deg (v) + deg (u) - 2 t (e) \\
&=& \sum_{v \in V (G)} deg^2 (v) - 2 \sum_{e \in E (G)} t (e) \\
&=& M_1 (G) - 6 t (G),
\end{eqnarray*}
since we counted each triangle three times. This completes the proof. \qed

\begin{remark}
In \cite{MoMaAs10}, the authors stated that
$$
n_u (e) + n_v (e) \geq deg (u) + deg (v) - t,
$$
where $t$ is a number such that every edge of $G$ lie in exactly $t$ triangles of $G$. Similarly as
in Proposition \ref{prp:PI-diameter}, this should be corrected to
$$
n_u (e) + n_v (e) \geq deg (u) + deg (v) - 2 t (e).
$$
After summing over all edges of $G$, we get
$$
P_v (G) \geq M_1 (G) - 6 t (G).
$$
\end{remark}

Similarly, we have the following

\begin{proposition}
\label{prp:Sz-diameter} Let $G$ be a connected graph with diameter $2$, such that every edge
belongs to exactly $t$ triangles. Then,
$$
Sz (G) = M_2 (G) - t \cdot M_1 (G) + m \cdot t^2.
$$
\end{proposition}

A graph $G = SRG (v, k, \lambda, \mu)$ is strongly regular if $G$ is $k$-regular $v$-vertex graph
such that every two adjacent vertices have $\lambda$ common neighbors, while every two non-adjacent
vertices have $\mu$ common neighbors. A simple corollary of Proposition \ref{prp:PI-diameter} and
Proposition \ref{prp:Sz-diameter} is~\cite{FaArMoAs10}
\begin{eqnarray*}
PI_v (SRG (v, k, \lambda, \mu)) &=& v k^2 - k v \lambda \\
Sz(SRG (v, k, \lambda, \mu)) &=& m k^2 - 2mk\lambda + m \lambda^2.
\end{eqnarray*}

The following upper bound on the vertex PI index was proved in \cite{DaGu10}.
\begin{theorem}
\label{thm:PI-triangle} Let $G$ be a connected graph with $n$ vertices and $m$ edges. Then,
$$
PI_v (G) \leq nm - 3 t (G).
$$
\end{theorem}

The authors moreover claimed, that the equality holds if an only if $G$ is a bipartite graph or $G
\cong K_3$. But this is not true. Namely, the equality holds if and only if for every edge $e = uv$
from $G$ holds $n_u (e) + n_v (e) = n - t (e)$. If $t (G) = 0$, then $PI_v (G) \leq nm$ if and only
if $G$ is a bipartite graph. Otherwise, the extremal graph must contain a triangle. For the
complete graph it holds
$$
n (n - 1) = PI_v (K_n) = n \cdot \binom{n}{2} - 3 \binom{n}{3}.
$$

We checked all graphs on $3 \leq n \leq 10$ vertices with the help of Nauty~\cite{Nauty}, and the
computational results are presented in Table 1, while the extremal graphs with $n = 4, 5, 6$
vertices are presented on Fig. 3. It is interesting that the diameter of all extremal graphs
(except $K_n$) is two.

\begin{table}[ht]
\centering
\begin{tabular} {l l l l l l l l l}
\toprule
$\mathbf{n}$ & 3 & 4 & 5 & 6 & 7 & 8 & 9 & 10 \\
\midrule
\textbf{count} & 1 & 2 & 4 & 7 & 11 & 17 & 25 & 36 \\
\bottomrule
\end{tabular}

\caption{The number of extremal non-bipartite small graphs.}
\end{table}

\begin{figure}[ht]
  \center
  \includegraphics [width = 15cm]{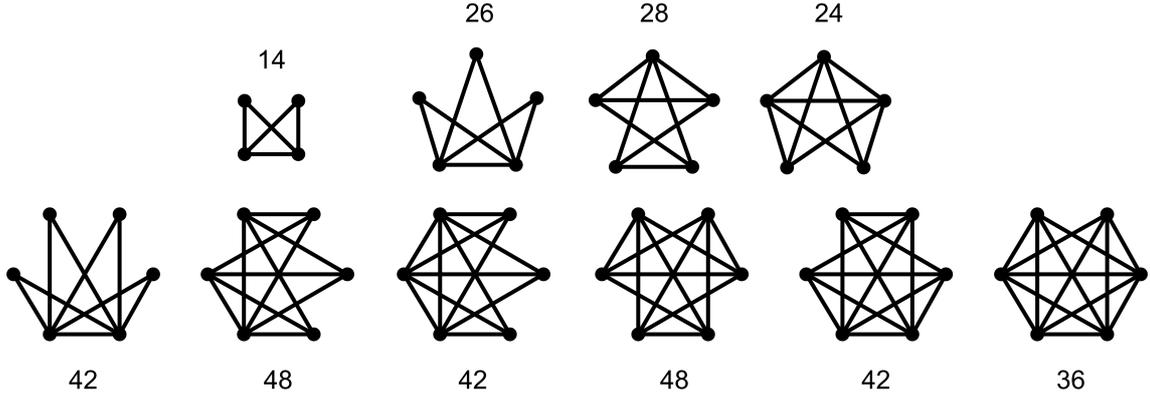}
  \caption { \textit{ Non-bipartite extremal graphs for $n = 4, 5, 6$ with $PI_v$ values. } }
\end{figure}

In addition, every extremal graph does not contain induced graph $C_3'$, composed of a triangle
with a pendent edge attached to one vertex of a triangle (otherwise the equality $n_u (e) + n_v (e)
= n - t (e)$ does not hold).

A graph is odd-hole-free if it has no induced subgraph that is a cycle of odd length greater than
3. Note that the equality holds in Theorem \ref{thm:PI-triangle} for odd-hole-free graphs.

A graph $G$ is a complete $k$-partite graph if there is a partition $V_1 \cup V_2 \cup \ldots \cup
V_k = V (G)$ of the vertex set, such that $uv \in E(G)$ iff $u$ and $v$ are in different parts of
the partition. If $|V_i| = n_i$, then $G$ is denoted by $K_{n_1,n_2,\ldots,n_k}$. It can be easily
proved that the equality also holds in Theorem \ref{thm:PI-triangle} for complete $k$-partite
graphs $K_{n_1,n_2,\ldots,n_k}$. Namely, consider an arbitrary edge $e = v_i v_j$ that connects
parts $V_i$ and $V_j$. The number of triangles that contain the edge $e$ is $n - n_i - n_j$, while
$n_{v_i} (e) = n_j$ and $n_{v_j} (e) = n_i$, and the relation $n_{v_i} (e) + n_{v_j} (e) = n - t
(e)$ holds.

For the completeness, we state the similar result for the Szeged index \cite{DaGu09}.

\begin{theorem}
Let $G$ be a connected graph with $n$ vertices, $m$ edges and $t (G)$ triangles. Then,
\begin{equation}
\label{eq:Sz-triangle} Sz (G) \leq \frac{1}{4} n^2m - 3 t (G).
\end{equation}
If equality holds in \eqref{eq:Sz-triangle}, then $G$ is bipartite (so in particular $t(G) = 0$),
regular, $n$ is even and the minimum vertex degree is greater than $1$.
\end{theorem}

A graph $G$ is distance-balanced if $|n_v (e)| = |n_u (e)|$ holds for any edge $e = uv$ of $G$ (see
\cite{JeKl08,KuMa06}). In \cite{IlKlMi10} it was proven that a connected bipartite graph $G$ is
distance-balanced if and only if $Sz (G) = \frac{1}{4}n^2 m$. Recently in \cite{ChWu09} the authors
presented a simple proof of the conjecture from \cite{KlRaGu96} that the complete balanced
bipartite graph $K_{\lfloor n/2 \rfloor, \lceil n/2 \rceil}$ has maximum Szeged index among all
connected graphs with $n$ vertices.

\begin{theorem}
Let $G$ be a connected triangle-free graph with $n \geq 3$ vertices. Then,
$$
\label{eq:sz-m2} Sz (G) \geq M_2 (G),
$$
with equality if and only if $G$ has diameter $2$.
\end{theorem}

\proof Let $e = uv$ be an arbitrary edge of $G$. Since $G$ is a triangle-free graph, it follows
$$
n_v (e) \cdot n_u (e) \geq deg (v) \cdot deg (u).
$$
Therefore,
$$
Sz (G) = \sum_{e \in E (G)} n_v (e) \cdot n_u (e) \geq \sum_{e \in E (G)} deg (v) \cdot deg (u) =
M_2 (G).
$$
Let $G$ be a triangle-tree graph such that $Sz (G) = M_2 (G)$. Since $G \not \cong K_n$, the
diameter of $G$ is greater than or equal to $2$. On the other side, if diameter is greater than
$2$, we can consider induced path $P_4 = w_0 w_1 w_2 w_3$ and for the edge $e = w_0 w_1$ it follows
that $n_{w_1} (e) > deg (w_1)$ since $d (w_1, w_3) = 2 < 3 = d (w_0, w_3)$. Therefore, $diam (G) =
2$. In this case, for all vertices $w$ that are on distance greater than one from the vertices $v$
and $u$ holds $d (v, w) = d (u, w) = 2$ and finally $n_v (e) \cdot n_u (e) = deg (v) \cdot deg
(u)$. This completes the proof. \qed

%%%%%%%%%%%%%%%%%%%%%%%%%%%%%%%%%%%%%%%%%%%%%%%%%%%%%%%
\section{Further relations}
\label{sec:4}
%%%%%%%%%%%%%%%%%%%%%%%%%%%%%%%%%%%%%%%%%%%%%%%%%%%%%%%

\begin{theorem}
Let $G$ be a connected graph with $n$ vertices and $m$ edges. Then,
$$
PI (G) \geq \frac{4}{m - 1} SZ_e (G),
$$
with equality if and only if $m$ is odd and $G \cong C_n$.
\end{theorem}

\proof Let $e = uv$ be an an arbitrary edge of $G$. Using the arithmetic-geometric mean inequality,
we have $(m_u (e) + m_v (e))^2 \geq 4 m_u (e) m_v (e)$. Therefore,
\begin{eqnarray*}
(m - 1) PI (G) &=& \sum_{e \in E (G)} (m - 1)(m_u (e) + m_v (e)) \\
&\geq& \sum_{e \in E (G)} (m_u (e) + m_v (e))^2 \\
&\geq& \sum_{e \in E (G)} 4 m_u (e) m_v (e) = 4 Sz_e (G).
\end{eqnarray*}
The equality holds if and only if $m_u (e) = m_v (e) = \frac{m - 1}{2}$ for every $e \in E (G)$. It
follows that $G$ does not have pendent vertices, and therefore $G$ is not a tree. Finally, $G$ must
contain a cycle, and let $C = v_1 v_2 \ldots v_k$ be the shortest cycle contained in $G$. If $k$ is
even, by considering the opposite edges we simply get $m_{v_i} (e) + m_{v_{i+1}} (e) < m - 1$. If
$k$ is odd and $G \not \cong C_n$, there exist a vertex $u$ not belonging to $C$, and without loss
of generality suppose that $u$ is a neighbor of $v_1$. In this case $d (v_{(k+1)/2}, u_1 v) = d
(v_{(k+3)/2}, u_1 v) = \frac{k-1}{2}$, and for the edge $e = v_{(k+1)/2} v_{(k+3)/2}$ we get
$m_{v_{(k+1)/2}} (e) + m_{v_{(k+3)/2}} (e) < m - 1$. Therefore, $k = n$ is odd number and $G \cong
C_n$. \qed

\begin{remark}
In \cite{FaArMoAs10}, the authors stated in Theorem 4 (b) that $PI (G) \geq \frac{4}{m-1} SZ_e (G)$
with equality if and only if $m - 1$ is even and $G$ is a tree with an odd number of vertices or a
cycle of odd length.
\end{remark}

We will establish another relation between Szeged and vertex PI index, using the following P\'
olya--Szeg\" o inequality \cite{PoSz72}.

%\begin{theorem}
%\label{thm:polya} Let $a_1, a_2, \ldots, a_n$ and $b_1, b_2, \ldots, b_n$ be positive real numbers
%such that for $1 \leq i \leq n$ holds $a \leq a_i \leq A$ and $b \leq b_i \leq B$. Then,
%$$
%\left ( \sum_{i = 1}^n a_i^2 \right ) \cdot \left ( \sum_{i = 1}^n b_i^2 \right ) \leq \frac{1}{4}
%\left ( \sqrt{\frac{AB}{ab}} + \sqrt{\frac{ab}{AB}} \right)^2 \cdot \left ( \sum_{i = 1}^n a_i b_i
%\right )^2.
%$$
%The equality holds if and only if the numbers
%$$
%p = \frac{\frac{A}{a}}{\frac{A}{a} + \frac{B}{b}} \cdot n  \qquad \mbox{and} \qquad q =
%\frac{\frac{ B}{b}}{\frac{A}{a} + \frac{B}{b}} \cdot n
%$$
%are integers, $a_1 = a_2 = \ldots = a_p = a$, $a_{p + 1} = a_{p + 2} = \ldots = a_n = A$, $b_1 =
%b_2 = \ldots = b_p = B$ and $b_{p+1} = b_{q+2} = \ldots = b_n = b$.
%\end{theorem}

\begin{theorem}
\label{thm:polya} Let $a_1, a_2, \ldots, a_n$ and $b_1, b_2, \ldots, b_n$ be positive real numbers
such that for $1 \leq i \leq n$ holds $a \leq a_i \leq A$ and $b \leq b_i \leq B$, with $a < A$ and
$b < B$. Then,
$$
\left ( \sum_{i = 1}^n a_i^2 \right ) \cdot \left ( \sum_{i = 1}^n b_i^2 \right ) \leq \frac{1}{4}
\left ( \sqrt{\frac{AB}{ab}} + \sqrt{\frac{ab}{AB}} \right)^2 \cdot \left ( \sum_{i = 1}^n a_i b_i
\right )^2.
$$
The equality holds if and only if the numbers
$$
p = \frac{\frac{A}{a}}{\frac{A}{a} + \frac{B}{b}} \cdot n  \qquad \mbox{and} \qquad q =
\frac{\frac{ B}{b}}{\frac{A}{a} + \frac{B}{b}} \cdot n
$$
are integers, $a_1 = a_2 = \ldots = a_p = a$, $a_{p + 1} = a_{p + 2} = \ldots = a_n = A$, $b_1 =
b_2 = \ldots = b_p = B$ and $b_{q+1} = b_{q+2} = \ldots = b_n = b$.
\end{theorem}

\begin{remark} By extending the proof of Theorem \ref{thm:polya}, if we allow $a = A$ or $b = B$, the
equality holds also if $AB = ab$, i. e. $a_1 = a_2 = \ldots = a_n = a = A$ and $b_1 = b_2 = \ldots
= b_n = b = B$.
\end{remark}

\begin{theorem}
Let $G$ be a simple graph with $n$ vertices and $m$ edges. Then
$$
32mn \cdot Sz(G) \leq (n+2)^2 \cdot PI_v^2 (G),
$$
with equality if and only if $m = 0$ or $n = 2$.
\end{theorem}

\begin{proof}
Using the arithmetic-geometric mean inequality, it follows
\begin{equation}
\label{eq:first} 4 \sum_{e \in E (G)} n_u (e) n_v (e) \leq \sum_{e \in E (G)} (n_u (e) + n_v
(e))^2.
\end{equation}
By setting in Theorem \ref{thm:polya} the values $a_i = 1$ and $b_i = n_u (e_i) + n_v (e_i)$ for $i
= 1, 2, \ldots, m$, we have
$$
\sum_{i = 1}^m 1^2 \cdot \sum_{e \in E (G)} (n_u (e) + n_v (e))^2 \leq  \frac{(AB + ab)^2}{4ABab}
\cdot \left ( \sum_{e \in E (G)} n_u (e) + n_v (e) \right )^2.
$$
Since $A = a = 1$, we need to estimate upper and lower bounds for $b_i$, $b = \min_{e \in E (G)}
n_u (e) + n_v (e) \geq 2$ and $B = \max_{e \in E (G)} n_u (e) + n_v (e) \leq n$. By analyzing the
function $x + \frac{1}{x}$ it follows
$$
\frac{(B + b)^2}{4Bb} = \frac{1}{4} \left ( \frac{B}{b} + \frac{b}{B} \right) + \frac{1}{2} \leq
\frac{(n + 2)^2}{8n}.
$$
Finally, we get
\begin{equation}
\label{eq:second} m \cdot \sum_{e \in E (G)} (n_u (e) + n_v (e))^2 \leq \frac{(n + 2)^2}{8n} \cdot
\left ( \sum_{e \in E (G)} n_u (e) + n_v (e) \right )^2.
\end{equation}
Combining Equations \eqref{eq:first} and \eqref{eq:second}, we complete the proof. The equality
holds if and only if $p = \frac{n}{1 + \frac{n}{2}} = \frac{2n}{n+2}$ is integer or $b = B$, which
is possible only for $n = 2$. Therefore, the equality holds if and only if $m = 0$ (in this case
$Sz (G) = PI_v (G) = 0$) or $n = 2$.
\end{proof}

\bigskip {\bf Acknowledgement. } This work was supported by Research
Grant 144007 of Serbian Ministry of Science. The author is grateful to Professor Sandi Klav\v zar
for his generous help and valuable suggestions that improved this article. The author is also
thankful to the anonymous referees for their remarks, which have improved the presentation of this
paper.

\end{document}